\documentclass[11pt,psamsfonts]{amsart}%
\usepackage{amsmath}
\usepackage{amsthm}
\usepackage{amssymb}
\usepackage{amscd}
\usepackage{amsfonts}
\usepackage{amsbsy}
\usepackage{graphicx}
\usepackage[dvips]{psfrag}%
\def\e{\varepsilon}

\newtheorem {theorem} {Theorem}

\newtheorem {corollary} [theorem] {Corollary}
\newtheorem {lemma} [theorem] {Lemma}

\newtheorem {remark} {Remark}
\begin{document}
\title[On the periodic solutions of a rigid dumbbell satellite]{On the periodic solutions of a rigid dumbbell satellite in a circular orbit}
\author[J.L.G. Guirao, J.A. Vera, B.A. Wade]{Juan L.G. Guirao$^{1}$, Juan A. Vera$^{2}$ and Bruce A. Wade$^{3}$}
\address{$^{1}$ Departamento de Matem\'{a}tica Aplicada y Estad\'{\i}stica. Universidad
Polit\'{e}cnica de Cartagena, Hospital de Marina, 30203--Cartagena, Regi\'{o}n
de Murcia, Spain.--Corresponding Author--}

\email{juan.garcia@upct.es}

\address{$^{2}$ Centro Universitario de la Defensa. Academia General del Aire.
Universidad Polit\'{e}cnica de Cartagena, 30720-Santiago de la Ribera,
Regi\'{o}n de Murcia, Spain.}

\email{juanantonio.vera@cud.upct.es}

\address{$^{3}$ Department of Mathematical Sciences,
University of Wisconsin--Milwaukee,
Milwaukee, WI 53201-0413, USA.}

\email{wade@uwm.edu}



\maketitle

\begin{abstract}
The aim of the present paper is to provide sufficient conditions for the existence of periodic
solutions of the perturbed attitude dynamics of a rigid dumbbell satellite in
a circular orbit.
\end{abstract}

\section{Introduction and statement of the main results}

\label{s1}
In this paper we consider the attitude dynamics, perturbed by small torques,
of a rigid dumbbell satellite in a circular orbit under the gravitational torque
of a central Newtonian force field. Our objective is to provide sufficient conditions for the existence of periodic motions
about the satellite's center of mass that are asymptotic to translational motion in an
absolute coordinate system. This type of motion,
denoted, \emph{cylindrical equilibrium}, is well known in the astrophysics literature
on satellite's dynamics (see for instance \cite{Guirao3, Vera2, Vera}).

These motions have an important application to satellite orientation problems because
a satellite can reach some specified nominal regime along periodic
trajectories only through the influence of gravitational torques and other
small perturbed torques induced by some control mechanism. 

\smallskip

Following the methods developed in \cite{Vera}, the equations of motion governing the attitude
dynamics of a rigid dumbbell satellite are
\begin{equation}%
\begin{array}
[c]{l}%
\dfrac{d^{2}\theta}{dt^{2}}-2\dfrac{d\phi}{dt}\left(  1+\dfrac{d\theta}%
{dt}\right)  \tan\phi+3\sin\theta\cos\theta=\varepsilon F_{1}^{\ast}\left(
t,\theta,\dfrac{d\theta}{dt},\phi,\dfrac{d\phi}{dt}\right)  ,\medskip\\
\dfrac{d^{2}\phi}{dt^{2}}+\left(  \left(  1+\dfrac{d\theta}{dt}\right)
^{2}+3\cos^{2}\theta\right)  \sin\phi\cos\phi=\varepsilon F_{2}^{\ast}\left(
t,\theta,\dfrac{d\theta}{dt},\phi,\dfrac{d\phi}{dt}\right)  ,
\end{array}
\label{dumb}%
\end{equation}
\medskip with $\theta$ and\textbf{ }$\phi$ the eulerian angles of nutation and
precession. The perturbed torques $F_{i}^{\ast}$, are smooth functions
periodic in the variable $t$ with
\[
F_{1}^{\ast}\left(  t,0,\dfrac{d\theta}{dt},0,\dfrac{d\phi}{dt}\right)
\equiv0\text{, }\ \text{ }F_{2}^{\ast}\left(  t,0,\dfrac{d\theta}{dt}%
,0,\dfrac{d\phi}{dt}\right)  \equiv0,
\]
\smallskip and $\varepsilon$ a small real parameter. In this work we are
interesed in the periodic functions emerging from the equilibrium solution
$\theta=0$ and $\phi=0$ of (\ref{dumb}) when $\varepsilon\rightarrow0.$

By means of the change of coordinates $x=\theta$ and $y=\phi$ and linearizing
the equations of (\ref{dumb}) in this equilibrium we obtain
\begin{equation}%
\begin{array}
[c]{c}%
\dfrac{d^{2}x}{dt^{2}}+3x=\varepsilon F_{1}\left(  t,x,\dfrac{dx}{dt}%
,y,\dfrac{dy}{dt}\right)  +\varepsilon^{2}R_{1}\left(  t,x,\dfrac{dx}%
{dt},y,\dfrac{dy}{dt},\varepsilon\right)  \medskip\\
\dfrac{d^{2}y}{dt^{2}}+4y=\varepsilon F_{2}\left(  t,x,\dfrac{dx}{dt}%
,y,\dfrac{dy}{dt}\right)  +\varepsilon^{2}R_{2}\left(  t,x,\dfrac{dx}%
{dt},y,\dfrac{dy}{dt},\varepsilon\right)
\end{array}
\label{e2}%
\end{equation}
with%
\[%
\begin{array}
[c]{c}%
F_{1}\left(  t,x,\dfrac{dx}{dt},y,\dfrac{dy}{dt}\right)  =f_{1}\left(
t,\dfrac{dx}{dt},\dfrac{dy}{dt}\right)  x+f_{2}\left(  t,\dfrac{dx}{dt}%
,\dfrac{dy}{dt}\right)  y\medskip\\
F_{2}\left(  t,x,\dfrac{dx}{dt},y,\dfrac{dy}{dt}\right)  =f_{3}\left(
t,\dfrac{dx}{dt},\dfrac{dy}{dt}\right)  x+f_{4}\left(  t,\dfrac{dx}{dt}%
,\dfrac{dy}{dt}\right)  y
\end{array}
\]
and $f_{i}$ smooth functions in the variables $\left(  t,\dfrac{dx}{dt}%
,\dfrac{dy}{dt}\right) .$ On the other hand the functions $f_{i}$ are
periodic in $t$ and in resonance $p$:$q$ with some of the periodic solutions
of the unperturbed dumbbell satellite given by
\begin{equation}
\begin{array}[c]{l}
\dfrac{d^{2}x}{dt^{2}}+3x=0,\vspace{0.2cm}\\
\dfrac{d^{2}y}{dt^{2}}+4y=0.
\end{array}
\label{e1a}
\end{equation}

The objective of this work is to provide, using as a main tool the averaging theory, a
system of nonlinear equations whose simple zeros provide periodic solutions of
the perturbed dumbbell satellite in circular orbit with equations of motion
given by (\ref{e2}). Some other works using similar techniques (see Llibre et al. \cite{BFL} for more details), are \cite{Guirao1,Guirao} or \cite{Guirao3} where the method is used for action angle variables. In order to present our results we need some preliminary
definitions and notations.

The unperturbed system \eqref{e1a} has a unique singular point, the origin
with eigenvalues
\[
\pm\sqrt{3}\,i,\quad\pm2\,i.
\]

Consequently this system in the phase space $\left(  x,\dfrac{dx}{dt}%
,y,\dfrac{dy}{dt}\right)  $ has two planes filled with periodic solutions with
the exception of the origin. These periodic solutions have periods
\[
T_{1}=\dfrac{2\pi}{\sqrt{3}}\quad\mbox{or}\quad T_{2}=\pi,
\]
according to whether they belong to the plane associated to the eigenvectors with
eigenvalues $\pm\sqrt{3}\,i$ or $\pm2\,i$, respectively. We shall study which
of these periodic solutions persist for the perturbed system \eqref{e2} when
the parameter $\varepsilon$ is sufficiently small and the perturbed functions
$F_{i}$ for $i=1,2$ have period either $pT_{1}/q$, or $pT_{2}/q$, where $p$
and $q$ are positive integers relatively prime.

We define the functions:%
\begin{equation}%
\begin{array}
[c]{l}%
\mathcal{F}_{1}(X_{0},Y_{0})=\displaystyle\dfrac{1}{2p\pi}%
%TCIMACRO{\dint _{0}^{pT_{1}}}%
%BeginExpansion
{\displaystyle\int_{0}^{pT_{1}}}
%EndExpansion
\sin\left(  \sqrt{3}t\right)  \Delta_{1}(t)\text{ }f_{1}\left(  t,\Delta
_{2}(t),0\right)  dt,\vspace{0.2cm}\\
\mathcal{F}_{2}(X_{0},Y_{0})=\displaystyle\dfrac{\sqrt{3}}{2p\pi}%
%TCIMACRO{\dint _{0}^{pT_{1}}}%
%BeginExpansion
{\displaystyle\int_{0}^{pT_{1}}}
%EndExpansion
\cos\left(  \sqrt{3}t\right)  \text{ }\Delta_{1}(t)\text{ }f_{1}\left(
t,\Delta_{2}(t),0\right)  dt,\vspace{0.2cm}\\
\mathcal{G}_{1}(Z_{0},W_{0})=\displaystyle\dfrac{1}{p\pi}%
%TCIMACRO{\dint _{0}^{pT_{2}}}%
%BeginExpansion
{\displaystyle\int_{0}^{pT_{2}}}
%EndExpansion
\sin\left(  2t\right)  \text{ }\Delta_{3}(t)\text{ }f_{4}\left(
t,0,\Delta_{4}(t)\right)  dt,\vspace{0.2cm}\\
\mathcal{G}_{2}(Z_{0},W_{0})=\displaystyle\dfrac{2}{p\pi}%
%TCIMACRO{\dint _{0}^{pT_{2}}}%
%BeginExpansion
{\displaystyle\int_{0}^{pT_{2}}}
%EndExpansion
\cos\left(  2t\right)  \text{ }\Delta_{3}(t)\text{ }f_{4}\left(
t,0,\Delta_{4}(t)\right)  dt,
\end{array}
\label{e3}%
\end{equation}
with
\[%
\begin{array}
[c]{l}%
\Delta_{1}(t)=X_{0}\cos\left(  \sqrt{3}t\right)  +\dfrac{Y_{0}}{\sqrt{3}}%
\sin\left(  \sqrt{3}t\right)  ,\vspace{0.2cm}\\
\Delta_{2}(t)=Y_{0}\cos\left(  \sqrt{3}t\right)  -\sqrt{3}X_{0}\sin\left(
\sqrt{3}t\right)  ,\vspace{0.2cm}\\
\Delta_{3}(t)=Z_{0}\cos\left(  2t\right)  +\dfrac{W_{0}}{2}\sin\left(
2t\right)  ,\vspace{0.2cm}\\
\Delta_{4}(t)=W_{0}\cos\left(  2t\right)  -2Z_{0}\sin\left(  2t\right)  .
\end{array}
\]

A zero $(X_{0}^{\ast},Y_{0}^{\ast})$ of the nonlinear system
\begin{equation}
\mathcal{F}_{1}(X_{0},Y_{0})=0,\quad\mathcal{F}_{2}(X_{0},Y_{0})=0, \label{e4}%
\end{equation}
such that
\[
\det\left(  \left.  \dfrac{\partial(\mathcal{F}_{1},\mathcal{F}_{2})}%
{\partial(X_{0},Y_{0})}\right\vert _{(X_{0},Y_{0})=(X_{0}^{\ast},Y_{0}^{\ast
})}\right)  \neq0,
\]
is called a \textit{simple zero} of system \eqref{e4}. Similarly, a zero
$(Z_{0}^{\ast},W_{0}^{\ast})$ of the nonlinear system
\begin{equation}
\mathcal{G}_{1}(Z_{0},W_{0})=0,\quad\mathcal{G}_{2}(Z_{0},W_{0})=0,
\label{e4bis}%
\end{equation}
such that
\[
\det\left(  \left.  \dfrac{\partial(\mathcal{G}_{1},\mathcal{G}_{2})}%
{\partial(Z_{0},W_{0})}\right\vert _{(Z_{0},W_{0})=(Z_{0}^{\ast},W_{0}^{\ast
})}\right)  \neq0,
\]
is called a \textit{simple zero} of system \eqref{e4bis}.

Our main results on the periodic solutions of the perturbed dumbbell satellite
\eqref{e2} are the following.

\begin{theorem}
\label{t1} Assume that the function $F_{1}^{\ast}$ and $F_{2}^{\ast}$ of the
perturbed dumbbell satellite with equations of motion (\ref{dumb}) are
periodic in $t$ of period $pT_{1}/q$ with $p$ and $q$ positive integers
relatively prime. Then, for $\varepsilon\neq0$ sufficiently small and for every
simple zero $(X_{0}^{\ast},Y_{0}^{\ast})\neq(0,0)$ of the nonlinear system
\eqref{e4}, the perturbed dumbbell satellite (\ref{dumb}) has a periodic
solution $\left(  \theta(t,\varepsilon),\dfrac{d\theta}{dt}(t,\varepsilon
),\phi(t,\varepsilon),\dfrac{d\phi}{dt}(t,\varepsilon)\right)  $ with%
\[
\lim_{\varepsilon\rightarrow0}\text{ }\left(  \theta(0,\varepsilon
),\dfrac{d\theta}{dt}(0,\varepsilon),\phi(0,\varepsilon),\dfrac{d\phi}%
{dt}(0,\varepsilon)\right)  =(X_{0}^{\ast},Y_{0}^{\ast},0,0).
\]

\end{theorem}

Theorem \ref{t1} is proved in section \ref{s3}. Its proof is based in the
averaging theory for computing periodic solutions, see the Appendix of the
paper.\smallskip

We provide an application of Theorem \ref{t1} in the following corollary,
which will be proved in section \ref{s4}.

\begin{corollary}
\label{c1} We consider the system (\ref{dumb}) with
\[%
\begin{array}
[c]{l}%
F_{1}^{\ast}\left(  t,\theta,\dfrac{d\theta}{dt},\phi,\dfrac{d\phi}%
{dt}\right)  =\sin\theta\left(  \dfrac{d\theta}{dt}\right)  ^{4}+\sin\phi
\sin\theta\left(  1-\left(  \dfrac{d\phi}{dt}\right)  ^{2}\right)  ,\medskip\\
F_{2}^{\ast}\left(  t,\theta,\dfrac{d\theta}{dt},\phi,\dfrac{d\phi}
{dt}\right)  =\cos\theta-\sin\left(  \sqrt{3}t\right)  \sin\theta
\dfrac{d\theta}{dt}-\sin\theta\left(  \dfrac{d\theta}{dt}\right)^{2}
\\ \qquad\qquad\qquad\quad\qquad-\cos\phi\left(  1-\left(  \dfrac{d\phi}{dt}\right)  ^{2}\right).
\end{array}
\]

\noindent Then, for $\varepsilon\neq0$ sufficiently small the system (\ref{dumb}) has one
periodic solution with
\[
\lim_{\varepsilon\rightarrow0}\text{ }\left(  \theta(0,\varepsilon
),\dfrac{d\theta}{dt}(0,\varepsilon),\phi(0,\varepsilon),\dfrac{d\phi}%
{dt}(0,\varepsilon)\right)  =\left(  \dfrac{\sqrt{3}}{3},0,0,0\right)  .
\]

\end{corollary}

Similarly we obtain the following result.

\begin{theorem}
\label{t1a} Assume that the functions $F_{1}^{\ast}$ and $F_{2}^{\ast}$ of the
perturbed dumbbell satellite with equations of motion (\ref{dumb}) are
periodic in $t$ of period $pT_{2}/q$ with $p$ and $q$ positive integers,
relatively prime. Then for $\varepsilon\neq0$ sufficiently small and for every
simple zero $(Z_{0}^{\ast},W_{0}^{\ast})\neq(0,0)$ of the nonlinear system
\eqref{e4bis}, the perturbed dumbbell satellite (\ref{dumb}) has a
periodic solution $\left(  \theta(t,\varepsilon),\dfrac{d\theta}%
{dt}(t,\varepsilon),\phi(t,\varepsilon),\dfrac{d\phi}{dt}(t,\varepsilon
)\right)  $ with%
\[
\lim_{\varepsilon\rightarrow0}\text{ }\left(  \theta(0,\varepsilon
),\dfrac{d\theta}{dt}(0,\varepsilon),\phi(0,\varepsilon),\dfrac{d\phi}%
{dt}(0,\varepsilon)\right)  =(0,0,Z_{0}^{\ast},W_{0}^{\ast}).
\]

\end{theorem}

The proof of Theorem \ref{t1a} is analogous 
Theorem \ref{t1}.

\smallskip

On the other hand, we provide an application of Theorem \ref{t1a} in the next
corollary, which will be proved in section \ref{s4}.

\begin{corollary}
\label{c2} We consider the system (\ref{dumb}) with%
\[%
\begin{array}
[c]{l}%
F_{1}^{\ast}\left(  t,\theta,\dfrac{d\theta}{dt},\phi,\dfrac{d\phi}%
{dt}\right)  =\sin\phi\sin\theta\dfrac{d\phi}{dt}+\sin\phi+\sin(2t)\sin
\phi\left(  1-\dfrac{d\phi}{dt}\right)  \dfrac{d\phi}{dt},\medskip\\
F_{2}^{\ast}\left(  t,\theta,\dfrac{d\theta}{dt},\phi,\dfrac{d\phi}%
{dt}\right)  =\sin\phi-\sin(2t)\sin\phi\dfrac{d\phi}{dt}-\sin\phi\left(
\dfrac{d\phi}{dt}\right)  ^{2}.
\end{array}
\]
Then, for $\varepsilon\neq0$ sufficiently small, the system (\ref{dumb}) has
three periodic solutions with%
\[%
\begin{array}
[c]{l}%
\lim\limits_{\varepsilon\rightarrow0}\text{ }\left(  \theta(0,\varepsilon
),\dfrac{d\theta}{dt}(0,\varepsilon),\phi(0,\varepsilon),\dfrac{d\phi}%
{dt}(0,\varepsilon)\right)  =\left(  0,0,1,\dfrac{2\sqrt{3}}{3}\right)
,\medskip\\
\lim\limits_{\varepsilon\rightarrow0}\text{ }\left(  \theta(0,\varepsilon
),\dfrac{d\theta}{dt}(0,\varepsilon),\phi(0,\varepsilon),\dfrac{d\phi}%
{dt}(0,\varepsilon)\right)  =\left(  0,0,\dfrac{1-\sqrt{17}}{4},0\right)
,\medskip\\
\lim\limits_{\varepsilon\rightarrow0}\text{ }\left(  \theta(0,\varepsilon
),\dfrac{d\theta}{dt}(0,\varepsilon),\phi(0,\varepsilon),\dfrac{d\phi}%
{dt}(0,\varepsilon)\right)  =\left(  0,0,\dfrac{1+\sqrt{17}}{4},0\right)  .
\end{array}
\]

\end{corollary}

\begin{remark}
The momenta apply to a rigid dumbbell satellite in a circular orbit under the gravitational torque
of a central Newtonian force field are, in general, functions of the Euler's angle and some maps which depend of the independent variable $t$ (time). The applications presented in Corollaries \ref{c1} and \ref{c2} represent some usual momenta which model the solar radiation obtained by the solar panels located at the satellite, see for more information \cite{Ma}.
\end{remark}

\section{Proof of Theorems \ref{t1} and \ref{t1a}}

\label{s3}

Introducing the variables $(X,Y,Z,W)=\left(  x,\dfrac{dx}{dt},y,\dfrac{dy}%
{dt}\right)  $ we write the differential system of the perturbed dumbbell
satellite \eqref{e2} as a first--order differential system defined in
$\mathbb{R}^{4}$, as follows:
\begin{equation}%
\begin{array}
[c]{l}%
\dfrac{dX}{dt}=Y,\medskip\\
\dfrac{dY}{dt}=-3X+\varepsilon F_{1}\left(  t,X,Y,Z,W\right)  +\varepsilon
^{2}R_{1}\left(  t,X,Y,Z,W,\varepsilon\right)  ,\medskip\\
\dfrac{dZ}{dt}=W,\medskip\\
\dfrac{dW}{dt}=-4Z+\varepsilon F_{2}\left(  t,X,Y,Z,W\right)  +\varepsilon
^{2}R_{2}\left(  t,X,Y,Z,W,\varepsilon\right)  .
\end{array}
\label{d1}%
\end{equation}
System \eqref{d1} with $\varepsilon=0$ is equivalent to the unperturbed
dumbbell satellite system \eqref{e1a}. On the other hand, the periodic orbits
of the unperturbed system are described in the following lemma.\smallskip

\begin{lemma}
\label{L1} The periodic solutions of the \textit{unperturbed system} with
$\varepsilon=0$ are
\begin{equation}%
\begin{array}
[c]{l}%
X(t)=X_{0}\cos\left(  \sqrt{3}t\right)  +\dfrac{Y_{0}}{\sqrt{3}}\sin\left(
\sqrt{3}t\right),\vspace{0.2cm}\\
Y(t)=Y_{0}\cos\left(  \sqrt{3}t\right)  -\sqrt{3}X_{0}\sin\left(  \sqrt
{3}t\right),\\
Z(t)=0,\\
W(t)=0,
\end{array}
\label{d3}%
\end{equation}
of period $T_{1}$, and
\begin{equation}%
\begin{array}
[c]{l}%
X(t)=0,\\
Y(t)=0,\\
Z(t)=Z_{0}\cos\left(  2t\right)  +\dfrac{W_{0}}{2}\sin\left(  2t\right),\vspace{0.2cm}\\
W(t)=W_{0}\cos\left(  2t\right)  -2Z_{0}\sin\left(  2t\right),
\end{array}
\label{d4}%
\end{equation}
of period $T_{2}$.
\end{lemma}

\begin{proof}
Since the unperturbed system is a linear differential system, the proof
is routine.
\end{proof}

\begin{proof}
[Proof of Theorem \ref{t1}]Assume that the function $F_{1}$ and $F_{2}$ of the
dumbbell satellite system with equations of motion \eqref{e2} are periodic in
$t$ of period $pT_{1}/q$ with $p$ and $q$ positive integers, relatively prime.
Then system \eqref{e2} is periodic in $t$ with period $pT_{1}$.

We shall apply Theorem \ref{tt} of the appendix to the differential system
\eqref{d1}. We note that system \eqref{d1} can be written as system
\eqref{eq:4}, taking
\[
\mathbf{x}=\left(
\begin{array}
[c]{c}%
X\vspace{0.2cm}\\
Y\vspace{0.2cm}\\
Z\vspace{0.2cm}\\
W
\end{array}
\right)  , G_{0}(t,\mathbf{x})=\left(
\begin{array}
[c]{c}%
\,Y\vspace{0.2cm}\\
-3\,X\vspace{0.2cm}\\
\,W\vspace{0.2cm}\\
-4Z
\end{array}
\right),
\]

$$
G_{1}(t,\mathbf{x})=\left(
\begin{array}
[c]{c}%
0\vspace{0.2cm}\\
F_{1}\left(  t,X,Y,Z,W\right)  \vspace{0.2cm}\\
0\vspace{0.2cm}\\
F_{2}\left(  t,X,Y,Z,W\right)
\end{array}
\right)$$
and
$$
G_{2}(t,\mathbf{x,\varepsilon})=\left(
\begin{array}
[c]{c}%
0\vspace{0.2cm}\\
\varepsilon^{2}R_{1}\left(  t,X,Y,Z,W,\varepsilon\right)  \vspace{0.2cm}\\
0\vspace{0.2cm}\\
\varepsilon^{2}R_{2}\left(  t,X,Y,Z,W,\varepsilon\right)
\end{array}
\right).$$

We shall study which periodic solutions \eqref{d3} of the unperturbed system
corresponding to the system \eqref{d1} with $\varepsilon=0$ can be continued
to periodic solutions of the unperturbed system for $\varepsilon\neq0$
sufficiently small.

\smallskip

We shall describe the different elements which appear in the statement of
Theorem \ref{tt}. Thus we have that $\Omega=\mathbb{R}^{4}$, $k=2$ and $n=4$.
Let $r_{1}>0$ be arbitrarily small and let $r_{2}>0$ be arbitrarily large. We
take the open and bounded subset $V$ of the plane $Z=W=0$ as
\[
V=\{(X_{0},Y_{0},0,0)\in\mathbb{R}^{4}:r_{1}<\sqrt{X_{0}^{2}+Y_{0}^{2}}%
<r_{2}\}.
\]
As usual $\mbox{\normalfont{Cl}}(V)$ denotes the closure of $V$. If
$\alpha=(X_{0},Y_{0})$, then we can identify $V$ with the set
\[
\{\alpha\in\mathbb{R}^{2}:r_{1}<||\alpha||\text{ }<r_{2}\},
\]
here $||\cdot||$ denotes the Euclidean norm of $\mathbb{R}^{2}$. The function
$\beta:\mbox{\normalfont{Cl}}(V)\rightarrow\mathbb{R}^{2}$ is $\beta
(\alpha)=(0,0)$. Therefore, in our case the set
\[
\mathcal{Z}=\left\{  \mathbf{z}_{\alpha}=\left(  \alpha,\beta(\alpha)\right)
,~~\alpha\in\mbox{\normalfont{Cl}}(V)\right\}  =\{(X_{0},Y_{0},0,0)\in
\mathbb{R}^{4}:r_{1}\leq\sqrt{X_{0}^{2}+Y_{0}^{2}}\leq r_{2}\}.
\]
Clearly for each $\mathbf{z}_{\alpha}\in\mathcal{Z}$ we can consider the
periodic solution $\mathbf{x}(t,\mathbf{z}_{\alpha})=(X(t),Y(t)$, $0,0)$ given
by \eqref{d3} of period $pT_{1}$.

\smallskip

Computing the fundamental matrix $M_{\mathbf{z}_{\alpha}}(t)$ of the linear
differential system with $\varepsilon=0$ associated to the $T$--periodic
solution $\mathbf{z}_{\alpha}=(X_{0},Y_{0},0,0)$ such that $M_{\mathbf{z}%
_{\alpha}}(0)$ is the identity of $\mathbb{R}^{4}$, we conclude that
$M(t)=M_{\mathbf{z}_{\alpha}}(t)$ is equal to
\[
\left(
\begin{array}
[c]{cccc}%
\cos\left(  \sqrt{3}\,t\right)  & \dfrac{\sin\left(  \sqrt{3}\,t\right)
}{\sqrt{3}} & 0 & 0\vspace{0.2cm}\\
-\sqrt{3}\sin\left(  \sqrt{3}\,t\right)  & \cos\left(  \sqrt{3}\,t\right)  &
0 & 0\vspace{0.2cm}\\
0 & 0 & \cos\left(  2\,t\right)  & \dfrac{\sin\left(  2\,t\right)  }{2}\vspace{0.2cm}\\
0 & 0 & -2\sin\left(  2\,t\right)  & \cos\left(  2\,t\right)
\end{array}
\right)  .
\]
Note that the matrix $M_{\mathbf{z}_{\alpha}}(t)$ does not depend of the
particular periodic solution $\mathbf{x}(t,\mathbf{z}_{\alpha})$. Since the
matrix
\[
M^{-1}(0)-M^{-1}(pT_{1})=\left(
\begin{array}
[c]{cccc}%
0 & 0 & 0 & 0\\
0 & 0 & 0 & 0\\
0 & 0 & 1-\cos\left(  \frac{4\sqrt{3}p\pi}{3}\right)  & \dfrac{\sin\left(
\frac{4\sqrt{3}p\pi}{3}\right)  }{2}\vspace{0.2cm}\\
0 & 0 & -\sin\left(  \frac{4\sqrt{3}p\pi}{3}\right)  & 1-\cos\left(
\frac{4\sqrt{3}p\pi}{3}\right)
\end{array}
\right)  ,
\]
satisfies the assumptions of statement (ii) of Theorem \ref{tt} because the
determinant
\[
\left\vert
\begin{array}
[c]{cc}%
1-\cos\left(  \frac{4\sqrt{3}p\pi}{3}\right)  & \dfrac{\sin\left(  \frac
{4\sqrt{3}p\pi}{3}\right)  }{2}\vspace{0.2cm}\\
-\sin\left(  \frac{4\sqrt{3}p\pi}{3}\right)  & 1-\cos\left(  \frac{4\sqrt
{3}p\pi}{3}\right)
\end{array}
\right\vert =4\sin^{2}\left(  \tfrac{2\sqrt{3}p\pi}{3}\right)  \neq0,
\]
we can apply this theorem to the unperturbed system.

\smallskip

Now $\xi:\mathbb{R}^{4}\rightarrow\mathbb{R}^{2}$ is $\xi(X,Y,Z,W)=(X,Y)$. We
calculate the function
\[
\mathcal{G}(X_{0},Y_{0})=\mathcal{G}(\alpha)=\xi\left(  \dfrac{1}{pT_{1}%
}\displaystyle\int_{0}^{pT_{1}}M_{\mathbf{z}_{\alpha}}^{-1}(t)G_{1}%
(t,\mathbf{x}(t,\mathbf{z}_{\alpha}))dt\right)  ,
\]
and we obtain
\[
\left(
\begin{array}
[c]{l}%
\mathcal{F}_{1}(X_{0},Y_{0})\vspace{0.2cm}\\
\mathcal{F}_{2}(X_{0},Y_{0})
\end{array}
\right)  =\left(
\begin{array}
[c]{c}%
\displaystyle\dfrac{1}{pT_{1}}%
%TCIMACRO{\dint _{0}^{pT_{1}}}%
%BeginExpansion
{\displaystyle\int_{0}^{pT_{1}}}
%EndExpansion
\dfrac{\sin\left(  \sqrt{3}\,t\right)  }{\sqrt{3}}F_{1}(t,\Delta_{1}%
(t),\Delta_{2}(t),0,0)dt\vspace{0.2cm}\\
\displaystyle\dfrac{1}{pT_{1}}%
%TCIMACRO{\dint _{0}^{pT_{1}}}%
%BeginExpansion
{\displaystyle\int_{0}^{pT_{1}}}
%EndExpansion
\cos\left(  \sqrt{3}\,t\right)  F_{1}(t,\Delta_{1}(t),\Delta_{2}(t),0,0)dt
\end{array}
\right)  .
\]
Using
\[
F_{1}(t,\Delta_{1}(t),\Delta_{2}(t),0,0)=\Delta_{1}(t)\text{ }f_{1}\left(
t,\Delta_{2}(t),0\right)
\]
we obtain the functions given by (\ref{e3}). Then, by Theorem \ref{tt} we have
that for every simple zero $(X_{0}^{\ast},Y_{0}^{\ast})\in V$ of the system of
nonlinear functions
\begin{equation}
\mathcal{F}_{1}(X_{0},Y_{0})=0,\quad\mathcal{F}_{2}(X_{0},Y_{0})=0, \label{hh}%
\end{equation}
we have a periodic solution $(X,Y,Z,W)(t,\varepsilon)$ of the unperturbed
system such that
\[
(X,Y,Z,W)(t,\varepsilon)\rightarrow(X_{0}^{\ast},Y_{0}^{\ast},0,0)\quad
\mbox{as $\e\to 0$.}
\]

\end{proof}

\section{Proof of the two corollaries}

\label{s4}

\begin{proof}
[Proof of Corollary \ref{c1}]Under the assumptions of Corollary \ref{c1} the
nonlinear equations \eqref{e4} becomes
\[%
\begin{array}
[c]{l}%
\mathcal{F}_{1}(X_{0},Y_{0})=\dfrac{Y_{0}(3X_{0}^{2}+Y_{0}^{2})}{48}%
,\vspace{0.2cm}\\
\mathcal{F}_{2}(X_{0},Y_{0})=\dfrac{\left(  3(\sqrt{3}-3X_{0})X_{0}^{2}%
-(\sqrt{3}+3X_{0})Y_{0}^{2}\right)  }{24}.
\end{array}
\]
This system has the following real solution%
\[
(X_{0}^{\ast},Y_{0}^{\ast})=\left(  \dfrac{\sqrt{3}}{3},0\right)
\]
Moreover%
\[
\det\left(  \left.  \dfrac{\partial(\mathcal{F}_{1},\mathcal{F}_{2})}%
{\partial(X_{0},Y_{0})}\right\vert _{_{(X_{0},Y_{0})=\left(  \frac{\sqrt{3}%
}{3},0\right)  }}\right)  =\frac{1}{384}%
\]
check that this solution is simple. So, by Theorem \ref{t1} we only have one
periodic solution of (\ref{dumb}). This completes the proof of the corollary.
\end{proof}

\begin{proof}
[Proof of Corollary \ref{c2}]Under the assumptions of Corollary \ref{c2} the
nonlinear equations \eqref{e4bis} becomes
\[%
\begin{array}
[c]{l}%
\mathcal{G}_{1}(Z_{0},W_{0})=\dfrac{(Z_{0}-1)W_{0}}{8},\vspace{0.2cm}\\
\mathcal{G}_{2}(Z_{0},W_{0})=\dfrac{\left(  4Z_{0}(2+Z_{0}-2Z_{0}^{2}%
)-W_{0}^{2}(1+2Z_{0})\right)  }{32}.
\end{array}
\]
This system has the following four real solutions
$$
(Z_{0}^{\ast},W_{0}^{\ast})=\left(  1,\pm\frac{2\sqrt{3}}{3}\right)  ,\text{
}\,(Z_{0}^{\ast},W_{0}^{\ast})=\left(  \frac{1-\sqrt{17}}{4},0\right)$$  

and

$$(Z_{0}^{\ast},W_{0}^{\ast})=\left(  \frac{1+\sqrt{17}}
{4},0\right)  .
$$
The solutions which differ in a sign are different initial conditions of the
same periodic solution of the system (\ref{e2}) and%
\[%
\begin{array}
[c]{l}%
\det\left(  \left.  \dfrac{\partial(\mathcal{G}_{1},\mathcal{G}_{1})}%
{\partial(Z_{0},W_{0})}\right\vert _{(Z_{0},W_{0})=\left(  1,\pm\frac
{2\sqrt{3}}{3}\right)  }\right)  =\dfrac{1}{32}\text{,\medskip}\vspace{0.2cm}\\
\det\left(  \left.  \dfrac{\partial(\mathcal{G}_{1},\mathcal{G}_{1})}%
{\partial(Z_{0},W_{0})}\right\vert _{(Z_{0},W_{0})=\left(  \frac{1-\sqrt{17}%
}{4},0\right)  }\right)  =\dfrac{17-7\sqrt{17}}{512},\text{\medskip}\vspace{0.2cm}\\
\det\left(  \left.  \dfrac{\partial(\mathcal{G}_{1},\mathcal{G}_{1})}%
{\partial(Z_{0},W_{0})}\right\vert _{(Z_{0},W_{0})=\left(  \frac{1+\sqrt{17}%
}{4},0\right)  }\right)  =\dfrac{17+7\sqrt{17}}{512}.
\end{array}
\]
Therefore, by Theorem \ref{t1a} we only have three periodic solutions of the
perturbed dumbbell satellite. This completes the proof of the corollary.
\end{proof}

\section*{Appendix: Basic results on averaging theory}

\label{ap}

In this appendix we present the basic result from the averaging theory that we
shall need for proving the main results of this paper.

\smallskip

We consider the problem of the bifurcation of $T$--periodic solutions from
differential systems of the form
\begin{equation}
\label{eq:4}\dot{\mathbf{x}}(t)= G_{0}(t,\mathbf{x})+\varepsilon
G_{1}(t,\mathbf{x})+\varepsilon^{2} G_{2}(t,\mathbf{x}, \varepsilon),
\end{equation}
with $\varepsilon=0$ to $\varepsilon\not = 0$ sufficiently small. Here the
functions $G_{0},G_{1}: \mathbb{R} \times\Omega\to\mathbb{R} ^{n}$ and
$G_{2}:\mathbb{R} \times\Omega\times(-\varepsilon_{0},\varepsilon_{0}%
)\to\mathbb{R} ^{n}$ are $\mathcal{C}^{2}$ functions, $T$--periodic in the
first variable, and $\Omega$ is an open subset of $\mathbb{R} ^{n}$. The main
assumption is that the unperturbed system
\begin{equation}
\label{eq:5}\dot{\mathbf{x}}(t)= G_{0}(t,\mathbf{x}),
\end{equation}
has a submanifold of periodic solutions.

\smallskip

Let $\mathbf{x}(t,\mathbf{z},\varepsilon)$ be the solution of the system
\eqref{eq:5} such that $\mathbf{x}(0,\mathbf{z},\varepsilon)=\mathbf{z}$. We
write the linearization of the unperturbed system along a periodic solution
$\mathbf{x}(t,\mathbf{z},0)$ as
\begin{equation}
\dot{\mathbf{y}}(t)=D_{\mathbf{x}}{G_{0}}(t,\mathbf{x}(t,\mathbf{z}%
,0))\mathbf{y}. \label{eq:6}%
\end{equation}
In what follows we denote by $M_{\mathbf{z}}(t)$ some fundamental matrix of
the linear differential system \eqref{eq:6}, and by $\xi:\mathbb{R}^{k}%
\times\mathbb{R}^{n-k}\rightarrow\mathbb{R}^{k}$ the projection of
$\mathbb{R}^{n}$ onto its first $k$ coordinates; i.e. $\xi(x_{1},\ldots
,x_{n})=(x_{1},\ldots,x_{k})$.

\smallskip

We assume that there exists a $k$--dimensional submanifold $\mathcal{Z}$ of
$\Omega$ filled with $T$--periodic solutions of \eqref{eq:5}. Then an answer
to the problem of bifurcation of $T$--periodic solutions from the periodic
solutions contained in $\mathcal{Z}$ for system \eqref{eq:4} is given in the
following result.

\begin{theorem}
\label{tt} Let $V$ be an open and bounded subset of $\mathbb{R} ^{k}$, and let
$\beta: \mbox{\normalfont{Cl}}(V)\to\mathbb{R} ^{n-k}$ be a $\mathcal{C}^{2}$
function. We assume that

\begin{itemize}
\item[(i)] $\mathcal{Z}=\left\{  \mathbf{z}_{\alpha}=\left(  \alpha,
\beta(\alpha)\right)  ,~~\alpha\in\mbox{\normalfont{Cl}}(V) \right\}
\subset\Omega$ and that for each $\mathbf{z}_{\alpha}\in\mathcal{Z}$ the
solution $\mathbf{x}(t,\mathbf{z}_{\alpha})$ of \eqref{eq:5} is $T$--periodic;

\item[(ii)] for each $\mathbf{z}_{\alpha}\in\mathcal{Z}$ there is a
fundamental matrix $M_{\mathbf{z}_{\alpha}}(t)$ of \eqref{eq:6} such that the
matrix $M_{\mathbf{z}_{\alpha}}^{-1}(0)- M_{\mathbf{z}_{\alpha}}^{-1}(T)$ has
in the upper right corner the $k\times(n-k)$ zero matrix, and in the lower
right corner a $(n-k)\times(n-k)$ matrix $\Delta_{\alpha}$ with $\det
(\Delta_{\alpha})\neq0$.
\end{itemize}

We consider the function $\mathcal{G} :\mbox{\normalfont{Cl}}(V) \to\mathbb{R}
^{k}$
\begin{equation}
\label{eq:7}\mathcal{G} (\alpha)=\xi\left(  \dfrac{1}{T} \int_{0}^{T}
M_{\mathbf{z}_{\alpha}}^{-1}(t)G_{1}(t,\mathbf{x}(t,\mathbf{z}_{\alpha}))
dt\right)  .
\end{equation}
If there exists $a\in V$ with $\mathcal{G} (a)=0$ and $\displaystyle{\det
\left(  \left(  {d\mathcal{G} }/{d\alpha}\right)  (a)\right)  \neq0}$, then
there is a $T$--periodic solution $\varphi(t,\varepsilon)$ of system
\eqref{eq:4} such that $\varphi(0,\varepsilon)\to\mathbf{z}_{a}$ as
$\varepsilon\to0$.
\end{theorem}

Theorem~\ref{tt} goes back to Malkin \cite{Ma} and Roseau \cite{Ro}, for a
shorter proof see \cite{BFL}.

\section*{Acknowledgements}
This work has been partially supported by MICINN/FEDER grant
number MTM2011--22587.

\end{document}